\documentclass[1p]{elsarticle}

\usepackage[plainpages=true]{hyperref}

\usepackage{natbib}

\newcommand\numberthis{\addtocounter{equation}{1}\tag{\theequation}}

\usepackage{amsfonts}
\usepackage{amsmath}
\usepackage{amssymb}
\usepackage{amsthm}

\newtheorem{Theorem}{Theorem}[section]
\newtheorem{Corollary}[Theorem]{Corollary}

\newtheorem{Lemma}[Theorem]{Lemma}

\newtheorem{Remark}{Remark}
\def\bbp{\mathbb{P}}
\def\bbe{\mathbb{E}}
\def\bbz{\mathbb{Z}}
\def\bbn{\mathbb{N}}

\def\bbr{\mathbb{R}}

\def\Cov{\operatorname{Cov}}
\def\Var{\operatorname{Var}}
\def\Poi{\operatorname{Poi}}

\begin{document}

\begin{frontmatter}

\title{Nonparametric Estimation of the Service Time Distribution in the Discrete-Time $GI/G/\infty$ Queue with Partial Information \tnoteref{t1}}
\tnotetext[t1]{NOTICE: this is the authors' version of a work that was accepted for publication in \emph{Stochastic Processes and their Applications}. Changes resulting from the publishing process, such as peer review, editing, corrections, structural formatting, and other quality control mechanisms may not be reflected in this document. Changes may have been made to this work since it was submitted for publication.}
\author[hd]{Sebastian Schweer\corref{cor1}}
\ead{schweer@uni-heidelberg.de}
\cortext[cor1]{Corresponding author}
\address[hd]{University of Heidelberg, Institute of Applied Mathematics, 69120 Heidelberg, Germany.}
\author[hd]{Cornelia Wichelhaus}
\ead{wichelhaus@statlab.uni-heidelberg.de}

\begin{abstract}
Estimation of the service time distribution in the discrete-time $GI/G/\infty$-queue based solely on information on the arrival and departure processes is considered. The focus is put on the estimation approach via the so called "sequence of differences". Existing results for this approach are substantially extended by proving a functional central limit theorem for the resultant estimator. Here, the underlying function space is taken to be the space of sequences converging to zero. The moving block bootstrap technique is considered for the estimation of the resultant covariance kernel and is shown to be applicable under mild additional conditions. 
\end{abstract}

\end{frontmatter}

\section{Introduction and Statement of the Problem}
\label{SectionIntroduction}

One general aim of research in the field of queueing theory is the improvement of the performance of given real life stochastic systems. In most of the practical applications, there are parameters or processes involved which are not available to the observer, hence there is an increasing interest in statistical inference depending only on incomplete information on the queueing models under consideration. In order to evaluate the performance of any given model or process, one of the main focuses of interest is the estimation of the service time distribution of the queue. In the case of queueing models with an infinite number of servers this problem has been addressed by several articles, in continuous time by e.g.~\cite{hall2004}, where the nonparametric analysis of the system was based on consecutive sequences of busy and idle periods, and by \cite{bingham}  where both parametric and nonparametric estimation is considered in a similar situation. Another approach was chosen in the early contributions \cite{brillinger74} and \cite{brown}, where nonparametric estimation based on observations of the arrival and departure processes is considered (also see \cite{wichelhaus2012}). For the analysis of discrete-time models we refer to \cite{pickands}, where the information about the queue is limited to counts of customers present in each time slot and \cite{dominic}, in which the nonparametric estimation of the service time distributions in the setting of queueing networks of general topology is discussed. For an extensive review on the existing literature including parametric approaches to the problem, we refer to \cite{wichelhaus2012}. 
 
In this paper, we consider the problem of estimating the cumulative distribution function (cdf) of the service time distribution $G$ in a discrete-time $GI/G/\infty$ queue, i.e.,~a queueing model with an infinite number of servers, a general service time distribution and a general i.i.d.~batch arrival process $(A(t))_{t\in \bbz}$. We assume that the available information about the behavior of this queue consists only of the counts of arrivals $(A(t))_{t\in \bbz}$ and departures $(D(t))_{t\in \bbz}$  from the queue in each time slot. More precisely, we consider a queue in which there is no possibility for the observer to distinguish any of the customers, so that the matching of any departure to its respective arrival is impossible. Additionally, the number of customers present at the beginning of the observation is also unavailable. Thus we are faced with a nonparametric estimation problem for the service time distribution $G$, for which we merely assume a finite mean and that its range is contained in $\bbn = \{1, 2, 3, \dots \}$.

We focus on the method of the so-called "sequence of differences", the underlying idea for which was introduced in \cite{brown} for the case of a continuous time $M/G/\infty$ queue. There, each departing customer is matched to the nearest arrival previous of his departure and the cdf of the resultant stationary sequence is estimated. Surprisingly, this cdf stands in a very simple relation to the sought after service time distribution. This result has also been discussed in \cite{Nelgabats2013}. Here, the sequence of differences $(Z(t))_{t \in \bbz}$ was generalized to the $r$-th sequence of differences $(Z^{r}(t))_{t \in \bbz}$, which does not give the time distance from a time of  departure $t_0$ to the first nearest arrival to the left of $t_0$, as $Z(t_0)$ does, but rather to the $r$-th  previous arrival to the left of $t_0$. Again, the resultant cdf of the  $(Z^{r}(t))_{t \in \bbz}$ can be shown to stand in an explicit relation to the cdf of the service time distribution, $G(\cdot)$.

The underlying principle was extended to the discrete-time case by \cite{dominic} via first defining the discrete-time \emph{sequence of differences} $\left( Z(t) \right)$ as
$$ Z(t):= t - \max\{n < t \ | \ A(n) > 0 \}, \quad t \in \bbz,
$$ which corresponds to the time elapsed since the most recent arrival for each time instant. As the next step, the following cdf $H(\cdot)$ and its estimator $ \widehat{H}_n  (\cdot)$ is defined for every $x \in \bbn$, 
\begin{equation}\label{DefinitionFunctionH}
H (x) := \frac{\bbe \left[  D(0) \mathbf{1}_{\{ Z (0)  \leq x \} } \right]}{\bbe[ D(0) ]} \text{  and  } \widehat{H}_n  (x) := \frac{  \sum_{i=1}^{n} D(i)  \mathbf{1}_{\{ Z (i)   \leq x \} } }{ \sum_{i=1}^{n} D(i)}.
\end{equation}
This distribution thus links a functional of the departure process with a functional of the arrival process. Given a realization of the arrival process $(A(t))_{t\in \{1, \dots, n\}}$ and the departure process $(D(t))_{t\in \{1, \dots, n\}}$, this distribution function can easily be established. Analogous to the continuous time model (cf. \cite[Lemma 2]{brown}), it can now be shown that there is a very simple relation linking the (easily accessible) distribution $H$ to the sought after service time distribution $G$. For the discrete-time case, \cite{dominic} showed that for every $x \in \bbn$,
\begin{equation}\label{PropRelationGH}
H(x) = 1 - c^{x}(1 - G(x)),
\end{equation} where $c := \bbp \left( A(0) = 0 \right)$. Defining the estimator $\hat{c}_n = \frac{1}{n} \sum_{i=1}^n \mathbf{1}_{ \{ A(i) = 0 \}}$, we obtain the following estimator for the service time distribution,
\begin{equation*}
\widehat{G}_n (x) := 1 - \hat{c}_n^{{-x}}\left(1 - \widehat{H}_n(x) \right).
\end{equation*}

Whereas \cite{brown} and \cite{Nelgabats2013} consider this estimation approach for a single node, \cite{dominic} extended the approach to queueing networks. However, in all three articles the asymptotic results for the estimator of $G$ do not go further than showing that the convergence $G_n (\cdot) \rightarrow G (\cdot)$ holds uniformly and almost surely. Obviously, this is a drawback for the practical applicability of this approach, since these results do not contain any information about the speed of convergence or about the limiting distribution.

In this paper, we amend this situation for the discrete-time case by showing that a functional central limit theorem for the estimator of $G$ holds. This result provides us with a rate of convergence as well as a limiting distribution of the estimator. For example, this result allows for the construction of confidence regions for certain values of $G$ or the entire distribution function based on a realization of the process. Further possible applications are given by goodness of fit tests, where the functional central limit theorem constitutes an important tool in establishing the limiting distribution of the test. \bigskip

We now state the main result of this paper. Let us first record some considerations which are necessary for the statement of the theorem: Since we are dealing with discrete distributions, it turns out to be unnecessary to consider convergence in the Skorokhod space $D[-\infty, \infty]$.  Following the example of \cite{henze1996}, we establish asymptotic theorems in the setting of the Banach space $c_0$ of all sequences $x = (x_k)_{k \in \bbn}$ converging to zero, equipped with the norm $\| x \|_{c_0} = \sup_{k \in \bbn} |x_k|$. For more details on this construction, we refer the reader to Section \ref{SectionFCLTandC0}.  Concerning notation, we denote the sequences associated with the cdf of the distributions $G$ and $H$ by $\mathcal{G} := \left( G(k) \right) _{k \in \bbn}$ and $\mathcal{H} := \left( H(k) \right) _{k \in \bbn}$. The respective estimators are denoted by $\mathcal{G}_n := ( \widehat{G}_n (k) )_{ k \in \bbn } $ and $\mathcal{H}_n :=( \widehat{H}_n (k) )_{ k \in \bbn }$.

\begin{Theorem}\label{TheoremAsymptoticNormailityG_n}
Let $\bbe[A(0)^2] < \infty$ and $\sum_{n=1}^{\infty} \sqrt{1 - G(n)} < \infty$. Then there exists a Gaussian sequence $\mathcal{V}  = \left( V_k \right) _{ k \in \bbn}$ in $c_0$ such that $\bbe\left[ V_k \right] = 0$ and 
\begin{align*}
\bbe\left[ V_k V_m \right] &= \frac{\tau^{2}_{k,m}}{c^{k+m}}  + km(1-H(k))(1-H(m)) \frac{1-c}{c^{k+m+1}} \\
 & - k (1- H(k)) \frac{\tau^{2}_{1,m}}{c^{k+m+1}} - m (1- H(m)) \frac{\tau^{2}_{1,k}}{c^{k+m+1}}
\end{align*} with 
\begin{align*}
\tau^{2}_{1,m} &= \frac{1}{\bbe[D(0)]} \sum_{i=0}^{\infty} \bbe \left[ D(i)   \left(   \mathbf{1}_{\{ Z(i) \leq m \} } - H(m) \right) \mathbf{1}_{ \left\{ A(0) = 0 \right\}} \right] \text{ and } \\
\tau^{2}_{k,m} &= \frac{1}{( \bbe[D(0)])^{2}} \sum_{i=- \infty}^{\infty} \bbe \Big[ D(0) D(i) \left( H(m) - \mathbf{1}_{ \left\{ Z(0) \leq m \right\}} \right) \left( H(k) - \mathbf{1}_{ \left\{ Z(i) \leq k \right\}} \right)\Big]
\end{align*}
 for $k,m \in \bbn$. Moreover,
$
 \sqrt{n} \left(  \mathcal{G}_n - \mathcal{G} \right) \stackrel{\mathcal{D}}{\rightarrow} \mathcal{V}
$ in $c_0$.
\end{Theorem}

A short discussion of the assumptions of this theorem: First, we remark that the condition $\sum_{n=1}^{\infty} \sqrt{1 - G(n)} < \infty$ is a condition on the tail behavior of the service time distribution $G$. As proven in Lemma \ref{LemmaRegularVarying} below, it can be replaced by a moment condition on the distribution $G$, i.e.~by assuming that $G$ has finite moments of order at least $2 + \epsilon$ for some $\epsilon > 0$. Thus, the moment conditions under which the result of Theorem \ref{TheoremAsymptoticNormailityG_n} hold are very mild. Additional to the condition on the tail behavior of $G$, we only assume finiteness of the second moment of the arrival distribution. This advantage comes at the cost of a somewhat more elaborate proof for tightness of the sequence of estimators of the cdf $H$, for details we refer to the proof of Theorem \ref{AsymptoticNormalityH_n}. Due to the mildness of conditions for the main result, it applies to a wide range of discrete-time queues with an infinite buffer size. For instance, the popular integer-valued auto-regressive models can be interpreted as $GI/G/\infty$-models satisfying the condition $\sum_{n=1}^{\infty} \sqrt{1 - G(n)} < \infty$, as the service time distribution is geometrical in this case. 

A different application is given by considering a discretized version of the continuous time $M/G/\infty$ estimation problem discussed in \cite{brown}. Suppose that we are given arrival and departure points of a continuous time $M/G/\infty$ process. We denote the sequence of arrival points by $\left\{A_{\text{cont.}}(t) \right\}_{t \in \bbz}$, governed by a Poisson process  of intensity $\lambda$, and the departure points by $\left\{D_{\text{cont.}}(t) \right\}_{t \in \bbz}$. Notice that $A_{\text{cont.}} (t), D_{\text{cont.}} (t) \in \bbr$ in this case. We now discretize the time domain with a certain step size $h > 0$ and define a discrete version of this process by simply setting $A_{\text{discr.}}(i) := \# \left\{  A_{\text{cont.}}(t) \in \left[ h(i-1), h i \right) | t \in \bbz \right\}$ and similarly for $D_{\text{discr.}} (t)$. Elementary properties of Poisson processes imply that the $A_{\text{discr.}}(i)$'s are i.i.d. according to a $ \Poi(\lambda  h)$ distribution. The assumption that the general (continuous) service time distribution $G$ satisfies the conditions of Theorem \ref{TheoremAsymptoticNormailityG_n}, or alternatively the moment condition imposed by Lemma \ref{LemmaRegularVarying}, implies that the same assumption holds for the discretized version $G_{\text{discr.}}$. Hence, we can apply our main result in this situation and obtain, for example, confidence bounds on the estimation of $G$. Notice that we may choose the parameter $h > 0$ arbitrarily small, thus ensuring that we can approximate the continuous $G$ arbitrarily well. Further, notice that the smaller $h>0$, the larger $c= \bbp \left( A(0) = 0 \right) = \exp( - \lambda  h)$ becomes, thus improving the asymptotic variance of the estimator given in Theorem \ref{TheoremAsymptoticNormailityG_n}. \bigskip

Let us give an overview of this article. The largest part of it is concerned with the proof of Theorem \ref{TheoremAsymptoticNormailityG_n}. We begin in Section \ref{SectionPrelims} by stating the precise model under consideration, introducing necessary notation and presenting preliminary results of interest. In Section \ref{SectionPointCLT} we prove finite-dimensional central limit theorems for the estimator of the distribution function $H(\cdot)$, first univariate then multivariate. These results allow us to present a proof for Theorem \ref{TheoremAsymptoticNormailityG_n} in the following Section \ref{SectionFLT}. In Section \ref{Bootstrap} we show that the bootstrapped version of the estimator discussed in the previous sections "works". This is important for the applicability of the results of this paper, since the resultant covariance kernel in Theorem \ref{TheoremAsymptoticNormailityG_n} is very involved and depends heavily on the unknown quantities $H$ and $c$, so that the bootstrap result at least allows for the computational approach to this problem.

\section{Preliminaries}
\label{SectionPrelims}

\subsection{The Model/Notation}
\label{SectionModelandNotation}

We let $\bbn = \{1, 2, \dots \}$ and $\bbn_0 = \bbn \cup \{0\}$. The behavior of the queue is modeled as follows: denote the number of arrivals in the  $t$-th time slot, the time slot between time $t$ and $t+1$, by $A(t)$ and the number of departures in this slot by $D(t)$. In each time slot $t \in \bbz$, indistinguishable customers labeled $K_{t,1}, \dots, K_{t,A(t)}$ arrive, where $A(t) =0$ is interpreted as no customers arriving in the $t$-th time slot. We assume that the sequence $(A(t))_{t\in \bbz}$ is i.i.d., has range $\bbn_0$ and that  $\bbe[A(0)] < \infty$. Each present customer $K_{k,j}$ receives upon arrival a sojourn time $S_{k,j}$ independently of all other customers arriving or present at the queue, where $S_{k,j}$ is distributed with cdf $G(\cdot)$, which has range $\bbn$ and a finite mean, i.e. $\sum_{i=1}^{\infty} (1 - G(i)) < \infty$. Denoting the probability masses of the distribution $G$ by $g_j$ for $ j \in \bbn$, we thus have $\bbp( S_{k,j} = l) = g_l$ for any $k \in \bbz$, $j,l \in \bbn$.  Each customer $K_{k,j}$ then remains in service exactly the number of time steps that his service time $S_{k,j}$ demands and then leaves the queue. We point out that we make the assumption $G(0)=0$ in order to ensure that each customer remains in the queue for at least one time step. We limit our knowledge about the considered system to the sequences $(A(t))_{t \in \bbz}$ and $(D(t))_{t \in \bbz}$ and we base our analysis of the behavior of this system solely on this information, i.e.~we do not assume to have any possibility of matching the arrival of certain customers to their respective departures. 

We define the "enlarged" process $(\xi (t))_{t \in \bbz}$ with $\xi(t) = \{ S_{t,1} \} \times \{S_{t,2} \} \times  \dots \times \{ S_{t,A(t)} \} $, the collection of all information given for the process in the $t$-th time slot, i.e., $\xi(t)$ carries information about both the number of arrivals in the $t$-th time slot as well as the service time distribution for these arrivals.  Notice that $\xi (t) \in \bbn^{\bbn} \cup \{0 \}$ for each $t \in \bbz$, where the state $0$ represents the case of no arrivals in the $t$-th time slot. Since the sequences of rvs $(S_{k,\cdot})_{k \in \bbz}$ and $(A(t))_{t \in \bbz}$ are i.i.d., it follows that the process $(\xi (t))_{t \in \bbz}$ is stationary and ergodic. We further define the $\sigma$-algebras
$
\mathcal{F}_k := \sigma \left( \xi (i) \ ; - \infty < i \leq k \right)$ 
 for $k \in \bbz$. With this construction, the process $(\xi (t))_{t \in \bbz}$ is an element of the space $\{ \bbn^{\bbn} \cup \{0 \} \}^{\bbz}$. We consider this Baire space to be endowed with its product topology, and we denote the Borel-$\sigma$-algebra based on the open sets of this topology by $\mathcal{F}_\infty$.

Under the assumption that the system has started in the infinite past, it follows from the construction of the process that we may express the queue length process $(Y(t))_{t\in \bbz}$, i.e. the number of customers in service during the $t$-th time slot by
\begin{equation*}
Y(t) = \sum_{j=0}^{\infty} \sum_{l=1}^{A(t-j)} \mathbf{1}_{\{ S_{t-j,l } > j \}},
\end{equation*}
where customers who leave during the $t$-th time slot are not considered to be in service. From this representation and the assumption that the sequences $(S_{k,\cdot})_{k \in \bbz}$ and $(A(t))_{t \in \bbz}$ are i.i.d., it follows that the queue length process is stationary. Further, the application of Wald's equation (see \eqref{LemmaWaldEquation} below) implies that under the assumption of a finite mean of both service time distribution and arrival distribution, the stationary distribution of the queue length process has a finite mean. Finally, we remark that we may express the departure process as follows
\begin{equation*}
D(t) = \sum_{j=1}^{\infty} \sum_{l=1}^{A(t-j)} \mathbf{1}_{\{ S_{t-j,l } = j \}},
\end{equation*} 
we will use this representation frequently throughout this article.

\subsection{Preliminary Results}

We begin by presenting some preliminary results needed during the course of this paper. First, a well-known result due to Wald as well as Blackwell and Girshick states that if $T, X_1, X_2, \dots$ are independent rvs with finite variance, and if $T$ has range $\bbn_0$ and the $X_1, X_2 ,\dots$ are identically distributed, then, with $ S_T := \sum_{i=1}^{T} X_i $,
\begin{equation}\label{LemmaWaldEquation}
 \bbe[S_T] = \bbe[T] \bbe[X_1] \text{ and } \Var \left( S_T \right)= \bbe[X_1]^{2} \Var(T) + \bbe[T] \Var(X_1).
\end{equation}
The former relation is called Wald's equation. As immediate consequences of these relations, we find that, for all $t \in \bbz$,
\begin{align*}
\bbe[D(t)] &= \sum_{j=1}^{\infty} \bbe \left[  \sum_{l=1}^{A(t-j)} \mathbf{1}_{\{ S_{t-j,l } = j \}} \right]  = \bbe[A(0)] \sum_{j=1}^{\infty} g_j = \bbe[A(0)].
\end{align*} We used that the sequence $\left( A(t) \right)_{t \in \bbz}$ is i.i.d.~and the monotone convergence theorem. Further, if $\bbe[A(0)^{2}] < \infty$,
\begin{align*}
\Var(D(0)) =  \sum_{j=1}^{\infty} \Var \left(  \sum_{k=1}^{A(-j)} \mathbf{1}_{\{ S_{-j,l} = j \} } \right) = \sum_{j=1}^{\infty} \left[\bbe[A(0)] g_j (1 - g_j) + g_j^{2} \Var(A (0)) \right].
\end{align*} Notice that $ \max \left\{ \bbe[A(0)], \Var(A(0))  \right\} \leq \max \left\{ \bbe[A(0)], \bbe[A(0)^{2}]  \right\}$, and since the rv $A(0)$ is discrete-valued, $\max \left\{ \bbe[A(0)], \bbe[A(0)^{2}]  \right\} = \bbe[A(0)^{2}]$. We may thus conclude that $\Var(D(0)) \leq 2 \bbe[A(0)^{2}]$ and $ \Var \left( D(0) \mathbf{1}_{\{ Z (0) \leq x \}} \right) \leq \bbe[D(0)^{2}] < \infty$.

\begin{Lemma}\label{LemmaRecursionZi} 
For $i \geq 1, k \geq 1$ and $j < i$ it holds that

\begin{align*}
a)& \quad \bbe \left[\mathbf{1}_{\{ Z (i) > k \} } \big| \mathcal{F}_0 \right] =  
\begin{cases}
c^{k}, &  i > k , \\
c^{i-1} \mathbf{1}_{ \{  Z (1) > k - i +1 \} }, & i \leq k, 
\end{cases}   \\
b)& \quad \bbe \left[  \mathbf{1}_{\{ Z (i) > k \} }  \sum_{l=1}^{A(i-j)} \mathbf{1}_{\{ S_{i-j,l } = j \}}    \Bigg| \mathcal{F}_0 \right] =  
\begin{cases}
0, &  j \in \{1, \dots, k \} , \\
\bbe[D(0)]  g_j  c^{k}, & j \in \{k+1, \dots, i-1 \}, 
\end{cases}  
\end{align*}
where the set $ \{k+1, \dots, i-1 \}$ is considered empty if $i \leq k$.
\end{Lemma}

\begin{proof}

From the definition of the random variable $Z (i)$ it is clear that we can write
\begin{align*}
Z (i) = \mathbf{1}_{\{ A(i-1) = 0 \} } \left( Z (i-1) + 1 \right) + \mathbf{1}_{\{ A(i-1) > 0 \} } = \mathbf{1}_{\{ A(i-1) = 0 \} }  Z (i-1) + 1.
\end{align*}  This implies
\begin{align*}
\mathbf{1}_{\{ Z (i) > k \} } = \mathbf{1}_{\{\mathbf{1}_{\{ A(i-1) = 0 \} }  Z (i-1)  > k - 1 \} }  = \mathbf{1}_{\{ A(i-1) = 0 \} }  \mathbf{1}_{\{ Z (i-1) > k-1 \} },
\end{align*} and with the tower rule for conditional expectations we find 
\begin{align*}
\bbe \left[\mathbf{1}_{\{ Z (i) > k \} } \big| \mathcal{F}_0 \right] = \bbe \left[ \bbe \left[ \mathbf{1}_{\{ A(i-1) = 0 \} }  \mathbf{1}_{\{ Z (i-1) > k-1 \} } \big| \mathcal{F}_{i-2} \right] \big| \mathcal{F}_0 \right] = c  \bbe \left[\mathbf{1}_{\{ Z (i-1) > k -1 \} } \big| \mathcal{F}_0 \right]
\end{align*} for $i \geq 2$, where the last equation used that $Z (i-1)$ is $\mathcal{F}_{i-2}$-measurable and $A(i-1)$ is independent of $\mathcal{F}_{i-2}$. The case $i=2$ is established directly. By definition, $\bbe \left[\mathbf{1}_{\{ Z (i) > 0 \} } \big| \mathcal{F}_0 \right] = 1$ for all $i \in \bbn$, and due to measurability, $\bbe \left[\mathbf{1}_{\{ Z (1) > k \} } \big| \mathcal{F}_0 \right] = \mathbf{1}_{\{ Z (1) > k \} }$ for all $k \in \bbn$. This allows us to prove relation a) for $i \geq 2$ recursively. For $i=1$ the statement is trivial.

To prove b), we first consider the case $j=1$ and $i > 2$. This expectation equals
\begin{align*}
&= \bbe \left[ \bbe \left[  \sum_{l=1}^{A(i-1)} \mathbf{1}_{\{ S_{i-1,l } = 1 \}}   \mathbf{1}_{\{ A(i-1) = 0 \} } \mathbf{1}_{\{ Z (i-1) > k-1 \} }  \Big| \mathcal{F}_{i-2}  \right] \Bigg| \mathcal{F}_0 \right] \\
& =  \bbe \left[  \sum_{l=1}^{A(i-1)} \mathbf{1}_{\{ S_{i-1,l } = 1 \}}   \mathbf{1}_{\{ A(i-1) = 0 \} }    \right] \bbe \left[  \mathbf{1}_{\{ Z (i-1) > k-1 \} }  \Big| \mathcal{F}_0 \right] = 0,
\end{align*}
as $Z (i-1)$ is $\mathcal{F}_{i-2}$-measurable and $A(i-1), S_{i-1, \cdot}$ are independent of $\mathcal{F}_{i-2}$. Just as in the proof of the first assertion, this argumentation can be extended recursively for all $j \in \{1, \dots, k \}$. Also as in the proof of the first assertion, the case $i =2$ follows directly, without invoking the tower rule. 

Now, if $j \in \{ k+1, \dots, i-1 \}$ then $i > k$. By the definition of $Z (i)$, the random variable $\mathbf{1}_{\{ Z (i) > k \} }$ depends only on the arrivals $A(i-1), \dots, A(i-k)$ and is independent of the random variables $A(i-k-1), \dots, A(1)$. Thus, we have 
\begin{align*}
&\bbe \left[ \mathbf{1}_{\{ Z (i) > k \} }   \sum_{l=1}^{A(i-j)} \mathbf{1}_{\{ S_{i-j,l } = j \}}    \Bigg| \mathcal{F}_0 \right] = \bbe \left[   \mathbf{1}_{\{ Z (i) > k \} }  \Big| \mathcal{F}_0 \right]  \bbe \left[  \sum_{l=1}^{A(i-j)} \mathbf{1}_{\{ S_{i-j,l } = j \}}  \right]. 
\end{align*}
As  $\bbe [  \sum_{l=1}^{A(i-j)} \mathbf{1}_{\{ S_{i-j,l } = j \}}  ] = \bbe[A(0)] g_j$ by Wald's equation and $\bbe[D(0)] = \bbe[A(0)]$, application of a) for the case $i> k$ thus proves b).
\end{proof}

The next result provides an upper bound for an expression which appears several times during the course of this article.

\begin{Lemma}\label{LemmaInequalityVariancesTechnical}
There exists a finite number $K$ such that for each $x \in \bbn$, $1 \leq i \leq x$ and $y \in \bbn_0$,  
\begin{align*}
&\Var \left(\left( c^{i-1} \mathbf{1}_{ \{  Z(1) > x - i +1 \} } -  c^{y}(1 - G(y)) \right) \sum_{j=i}^{\infty} \sum_{l=1}^{A(i-j)} \mathbf{1}_{\{ S_{i-j,l } = j \}} \right) \\
&\leq \left[c^{2y} (1 - G(y))^{2} - 2 c^{x+y} (1 - G(y)) + c^{x+i-1}  \right] \frac{\bbe[A(0)^{2}]  K}{1-c} (1 - G(i-1)).
\end{align*}
if this variance is finite. 
\end{Lemma}

\begin{proof}
We simplify the notation in this proof by setting  $r_j := \sum_{l=1}^{A(i-j)} \mathbf{1}_{\{ S_{i-j,l } = j \}}$ and $R_i := \sum_{j=i}^{\infty} r_j$, recall that the second moment of $R_i$ is finite by \eqref{LemmaWaldEquation}.  Let $z \in \bbn$ be arbitrary, then application of the law of total probability yields 
$\bbe[r_{i+z}^{q} | A(-z) > 0] = \frac{1}{1-c} \bbe[r_{i+z}^{q}]$  for $q \in \bbn$. For $q=1$, $\bbe[r_{i+z}^{q}] = \bbe[A(0)] g_{i+z}$ by \eqref{LemmaWaldEquation}. For $q=2$, we use \eqref{LemmaWaldEquation} and the inequalities established in the discussion following that expression for  
\begin{align}\label{eqVarianceInequalityWald}
&\Var \left( r_{i+z} \right) = \bbe[A(0)] (g_{i+z} - g_{i+z}^{2})  + g_{i+z}^{2} \Var(A(0)) \leq \bbe[A(0)^{2}] g_{i+z}.
\end{align}
We obtain
\[ \bbe\left[ r_{i+z}^{2}  \right] \leq \bbe[A(0)^{2}] g_{i+z} +  \bbe[A(0)]^{2} g_{i+z}^{2} \leq 2 \bbe[A(0)^{2}] g_{i+z},
\]
since Jensen's inequality implies $\bbe[A(0)]^{2} \leq \bbe[A(0)^{2}]$. The same argumentation together with the independence of  $(A_t)_{t \in \bbz}$ implies  $\bbe[R_{i+z+1}^{2}] \leq 2 \bbe[A(0)^{2}] (1 - G(i+z))$. Now, the event $Z(1) = z$ entails that $A(0) = \dots = A(-z+1)=0$ and $A(z) > 0$. Combining all of these results with the linearity of the expectation, we find
\begin{align*}
 &\bbe\left[ R_i^{2} | Z(1) = z \right] = \bbe\left[ R_{i+z+1}^{2} \right] + 2 \bbe\left[R_{i+z+1} \right] \bbe[r_{i+z} | A(-z) > 0 ] +  \bbe[r_{i+z}^{2} | A(-z) > 0 ]\\
 &\leq  2 \bbe[A(0)^{2}] (1 - G(i+z)) + \frac{2 \bbe[A(0)]^{2} (1 - G(i+z)) g_{i+z}}{1-c}  + \frac{2 \bbe[A(0)^{2}] g_{i+z}}{1-c}.
\end{align*}
As $G(\cdot)$ is a cdf it is monotonously increasing in its argument, there exists a positive constant $K$ such that this expression is bounded by $(K \bbe[A(0)^{2}] (1 - G(i-1))/(1-c)$. This bound holds for all $z \in \bbn$ and can thus be extended to conditions of the form $ \{ z \in B \subseteq \bbn \}$ by the law of total probability. Setting $A_0 = \left\{ Z(1) > x-i+1 \right\}$ and $ A_1 = \left\{ Z(1) \leq x-i+1 \right\}$, we first have $\bbp(A_0) = c^{x-i+1}$. The assertion is now an easy consequence of the inequality $\Var(R_i) \leq \bbe[R_i^{2}]$ and the law of total probability, i.e. $\bbe[R_i^{2}] = \bbp(A_0) \bbe[R_i^{2} | A_0] + \bbp(A_1) \bbe[R_i^{2} | A_1]$.
\end{proof}

We now consider the ergodicity of the sequence $ \left(D(i) \mathbf{1}_{\{ Z (i)   \leq x \} } \right)_{ i \in \bbz}$, which is stationary due to the model assumptions made for the i.i.d.~sequences $(S_{k,\cdot})_{k \in \bbz}$ and $(A(t))_{t \in \bbz}$. The following result was shown in Lemma 2 in \cite{dominic}, the proof follows in the same vein as that of Lemma 1 in \cite{brown} and Proposition 3 in \cite{Nelgabats2013}.

\begin{Lemma}\label{TheoremMeasurable}
The sequences $ \left(D(i) \mathbf{1}_{\{ Z (i)   \leq x \} } \right)_{ i \in \bbz}$ and $ \left(D(i)\right)_{ i \in \bbz}$ are stationary and ergodic.
\end{Lemma} 

As a first consequence of Lemma \ref{TheoremMeasurable}, Birkhoff's ergodic theorem yields, for all $x \in \bbn$,
\[ \frac1n \sum_{i=1}^{n}  D(i) \mathbf{1}_{\{ Z (i)   \leq x \} }  \rightarrow \bbe \left[ D(0) \mathbf{1}_{\{ Z (0)  \leq x \} } \right]  \text{ a.s.},
\] so that $\widehat{H_n } (x) \rightarrow H  (x)$ a.s. The relation between $G(x)$ and $H(x)$ as well as between $ \widehat{G_n } (x)$ and $ \widehat{H_n } (x)$ is a continuous one (see \eqref{PropRelationGH}). Further, $\hat{c}_n \rightarrow c$ a.s. is obvious since the sequence $(A(t))_{t \in \bbz}$ is i.i.d.~and we assumed $\Var(A(0)) < \infty$. We use the continuous mapping theorem to find 
\[ \widehat{G_n } (x) \rightarrow G  (x), \quad \text{a.s.},
\] which holds pointwise for all $x \in \bbn$. In particular, it follows that $\widehat{G_n } (x) \rightarrow 1$ for $x \to \infty$.

\section{Finite Dimensional Central Limit Theorems}
\label{SectionPointCLT}

In this section we prove finite dimensional central limit theorems (CLT's) for the estimator of the distribution function $H$. Due to the special structure of these estimators, see \eqref{DefinitionFunctionH}, it is necessary to first show CLT's for the numerator and the denominator and then combining these results with an appropriate expansion of the terms. The former results are shown in Theorem \ref{NormalityD(i)1} and Corollary \ref{NormalityD(i)}, respectively, the latter in Theorem \ref{TheoremAsymptoticVarianceHn}.

In what follows, we will use Theorem 19.1.~in \cite{billingsley1999convergence} repeatedly. There, a CLT is shown for stationary and ergodic sequences $(X_t)_{t \in \bbz}$ under the conditions of a finite second moment of the $X_i$'s and that
\begin{equation}\label{eqBillingsleyCondition}
\sum_{j=1}^\infty \left\|  \bbe \left[  X_j - \bbe[X_0] \  \big|  \ \sigma\left( \left\{X_i\right\}_{ -\infty < i \leq 0}  \right)  \right] \right\| < \infty,
\end{equation} where $\| \cdot \|$ denotes the $L^2$ norm and $\sigma( \{X_i\}_{ -\infty < i \leq 0} )$ denotes the $\sigma$-algebra generated by the rvs $X_i$ for $i \leq 0$. This condition can be seen as a form of mixing condition, in the sense that, under this condition, the dependence of the process on a given reference state $0$ decays fast enough to $0$ to be summable in the $L^2$ norm. The advantage of this condition in comparison to the better known mixing conditions is that the conditional expectation is more easily accessible and thus \eqref{eqBillingsleyCondition} can be shown to hold for both sequences involved in the estimator \eqref{DefinitionFunctionH}. We refer to the proofs of the following theorems for details. We remark that the establishment of classical mixing conditions in the framework of the continuous time queuing model proved elusive for Brown, as he states in his paper (cf. \cite[p. 653]{brown}) that he "has been unable to verify the mixing conditions given by Billingsley [...]". We point out that he referred here to the first edition of \cite{billingsley1999convergence}, whereas the condition we apply here was published 30 years later, in the second edition of this book.

\begin{Theorem}\label{NormalityD(i)1}
Let the conditions of Theorem \ref{TheoremAsymptoticNormailityG_n} be satisfied. Then, for each $x \in \bbn$,
\begin{equation*}
 \sqrt{n}  \left(   \frac1n \sum_{i=1}^{n} D(i) \mathbf{1}_{\{ Z(i) \leq x \} } - \bbe \left[  D(0) \mathbf{1}_{\{ Z(0) \leq x \} } \right]  \right) \stackrel{\mathcal{D}}{\rightarrow} \mathcal{N} (0 , \sigma_x^{2}),
\end{equation*}
where
$
\sigma_x^{2} =  \Var \left( D(0)  \mathbf{1}_{\{ Z(0) \leq x \} } \right) + 2 \sum_{j=1}^{\infty}  \Cov \left(D(0)  \mathbf{1}_{\{ Z(0) \leq x \} } , D(j)  \mathbf{1}_{\{ Z(j) \leq x \} } \right) .
$
\end{Theorem}

\begin{proof}
Let $x \in \bbn$. By Lemma \ref{TheoremMeasurable}, the sequence $ \left(D(i) \mathbf{1}_{\{ Z(i) \leq x \} } \right)_{ i \in \bbz}$ is stationary and ergodic. A direct implication of \eqref{LemmaWaldEquation} is that the rvs $D(i) \mathbf{1}_{\{ Z(i) \leq x \} } $ have finite second moments. It remains to be seen that the condition \eqref{eqBillingsleyCondition} is satisfied. We remark that $ \sigma \left( D(i) \mathbf{1}_{\{ Z(i) \leq x \} } ; i \leq k \right) \subset \sigma \left( \xi(i) ; i \leq k \right) = \mathcal{F}_k$,  
where the $\xi(i)$'s are the enlarged process defined in Section \ref{SectionModelandNotation}. Using (19.25) in \cite{billingsley1999convergence}, it suffices to show that
\begin{equation}\label{ConditionSummable}
\sum_{i =1 }^{\infty} \left\| \bbe \left[ D(i) \mathbf{1}_{\{ Z(i) \leq x \} } - \bbe \left[  D(0) \mathbf{1}_{\{ Z(0) \leq x \} } \right] \Big| \mathcal{F}_0 \right] \right\| < \infty, 
\end{equation}
we remark that $\bbe \left[  D(0) \mathbf{1}_{\{ Z(0) \leq x \} } \right] = \bbe \left[  D(i) \mathbf{1}_{\{ Z(i) \leq x \} } \right]$ due to stationarity. 

Consider the case $i > x$. In the first step, we separate the random variable given by the conditional expectation in the expression above into its probabilistic and deterministic parts. For instance, the arrivals occurring after the time slot $0$ are independent of $\mathcal{F}_0$ and, since $ i > x$, so is the random variable $\mathbf{1}_{\{ Z(i) \leq x \}}$. Similarly,  since the random variable $\mathbf{1}_{\{ Z(i) \leq x \}}$ is independent of the behavior of the process before time $i-x > 0$, we have
\begin{align*}
&\bbe \left[ D(i) \mathbf{1}_{\{ Z(i) \leq x \} } \Big| \mathcal{F}_0 \right] - \bbe \left[ D(i) \mathbf{1}_{\{ Z(i) \leq x \} } \right] \\
&= \bbe \left[  \mathbf{1}_{\{ Z(i) \leq x \} } \right] \sum_{j=i}^{\infty} \sum_{l=1}^{A(i-j)} \mathbf{1}_{\{ S_{i-j,l } = j \}} -  \bbe \left[ \mathbf{1}_{\{ Z(i) \leq x \} } \right] \bbe \left[ \sum_{j=i}^{\infty} \sum_{l=1}^{A(i-j)} \mathbf{1}_{\{ S_{i-j,l } = j \}}  \right].
\end{align*}
Thus, as $\bbe [ \mathbf{1}_{\{ Z(i) \leq x \} } ] = 1 - c^{x}$,
\begin{align*}
&\left\| \bbe \left[ D(i) \mathbf{1}_{\{ Z(i) \leq x \} } \Big| \mathcal{F}_0 \right] - \bbe \left[  D(i) \mathbf{1}_{\{ Z(i) \leq x \} } \right]  \right\| =  (1-c^{x}) \Var^{\frac12}\left(   \sum_{j=i}^{\infty} \sum_{l=1}^{A(i-j)} \mathbf{1}_{\{ S_{i-j,l } = j \}}  \right).
\end{align*}
Since the random variables $S_{i-j_1, \cdot}$ and $S_{i-j_2, \cdot}$ are independent for $j_1 \neq j_2$, we find with \eqref{eqVarianceInequalityWald}
\begin{align*}
&\left\| \bbe \left[ D(i) \mathbf{1}_{\{ Z(i) \leq x \} } - \bbe \left[  D(0) \mathbf{1}_{\{ Z(0) \leq x \} } \right] \Big| \mathcal{F}_0 \right] \right\|  \leq \left( 1 - c^{x} \right) \sqrt{\bbe[A(0)^{2}]} \sqrt{1 - G(i-1)}.
\end{align*}

Now, let us consider the case $ i \leq x$. We will use a similar approach as above, separating $\bbe \left[ D(i) \mathbf{1}_{\{ Z(i) \leq x \} } \Big| \mathcal{F}_0 \right]$ in its probabilistic and deterministic parts. First,  
\begin{align*}
&\bbe \left[ D(i) \mathbf{1}_{\{ Z(i) \leq x \} } \Big| \mathcal{F}_0 \right] \\
&= \sum_{j=i}^{\infty} \sum_{l=1}^{A(i-j)} \mathbf{1}_{\{ S_{i-j,l } = j \}} \bbe \left[  \mathbf{1}_{\{ Z(i) \leq x \} } \big| \mathcal{F}_0 \right] + \bbe \left[ \sum_{j=1}^{i-1} \sum_{l=1}^{A(i-j)} \mathbf{1}_{\{ S_{i-j,l } = j \}} \mathbf{1}_{\{ Z(i) \leq x \} } \Big| \mathcal{F}_0 \right] \\
&= \sum_{j=i}^{\infty} \sum_{l=1}^{A(i-j)} \mathbf{1}_{\{ S_{i-j,l } = j \}} \left( 1 - c^{i-1} \mathbf{1}_{\{ Z(1) > x -i +1  \} } \right) + \bbe \left[ \sum_{j=1}^{i-1} \sum_{l=1}^{A(i-j)} \mathbf{1}_{\{ S_{i-j,l } = j \}}   \right],
\end{align*}
the second equality used Lemma \ref{LemmaRecursionZi} and the fact that the rvs $A(i), S_{i,l}$ are independent of $\mathcal{F}_0$ for $i > 0$. Since the tower rule for conditional expectations implies that $\bbe[ \bbe[ D(i) \mathbf{1}_{\{ Z(i) \leq x \} } | \mathcal{F}_0 ] ] = \bbe[D(i) \mathbf{1}_{\{ Z(i) \leq x \} }]$, it follows that
\begin{align*}
&\left\| \bbe \left[ D(i) \mathbf{1}_{\{ Z(i) \leq x \} } \Big| \mathcal{F}_0 \right] - \bbe \left[  D(i) \mathbf{1}_{\{ Z(i) \leq x \} } \right]  \right\| \\
&= \Var^{\frac12} \left( \sum_{j=i}^{\infty} \sum_{l=1}^{A(i-j)} \mathbf{1}_{\{ S_{i-j,l } = j \}} \left( 1 - c^{i-1} \mathbf{1}_{\{ Z(1) > x -i +1  \} } \right)  \right) ,
\end{align*}
for which we find an upper bound using Lemma \ref{LemmaInequalityVariancesTechnical} and setting $y = 0$, notice that $G(0) = 0$. We are now able to combine the results for the cases $i > x$ and $i \leq x$. Changing the summation index for convenience, we obtain an upper bound for \eqref{ConditionSummable}
\begin{align*}
&\sqrt{\bbe[A(0)^{2}]} \left[ \sum_{i=0}^{x-1} \sqrt{ \left( 1 - 2 c^{x} + c^{x+i} \right) \frac{ K}{1-c} (1 - G(i))} +\left( 1 - c^{x} \right)  \sum_{i=x}^{\infty}  \sqrt{1 - G(i)} \right]. 
\end{align*}
Since $ 1 - 2 c^{x} + c^{x+i} \leq  1$  for all $0 \leq i \leq x-1$ this expression has an upper bound
$
\sqrt{ \frac{\bbe[A(0)^{2}] K}{1-c}} \sum_{i=0}^{\infty}\sqrt{1 - G(i)}
$
and since $ \sum_{i=0}^{\infty} \sqrt{1 - G(i)} < \infty$ by assumption, this concludes the proof.
\end{proof}

\begin{Remark}
We point out that the inequality \eqref{eqVarianceInequalityWald}, which is used at a crucial part of the proof of Theorem \ref{NormalityD(i)1}, is not a very rough estimate. Apart from the upper bound on the moments, it can not be improved upon without a loss of generality. To see this, consider the very common assumption of Poisson($\lambda$) distributed arrivals. In this case, the resultant variance in \eqref{eqVarianceInequalityWald} actually equals $\lambda g_j$ rendering the upper bound found for \eqref{ConditionSummable} sharp.
\end{Remark}

The following result is an easy consequence of Theorem \ref{NormalityD(i)1}, it is achieved by setting  $x = \infty$ and using inequality \eqref{eqVarianceInequalityWald} again.

\begin{Corollary}\label{NormalityD(i)}
Let the conditions of Theorem \ref{TheoremAsymptoticNormailityG_n} be satisfied. Then 
\begin{equation*}
 \sqrt{n} \left(  \frac1n  \sum_{i=1}^{n} D(i) - \bbe[ D(0) ] \right) \stackrel{\mathcal{D}}{\rightarrow} \mathcal{N} (0 , \sigma^{2}), 
\end{equation*}
where
$
\sigma^{2} 
 = \Var(D(0)) + 2 \sum_{j=1}^{\infty}  \Cov(D(0), D(j)). 
$
\end{Corollary}

Concerning the condition $\sum_{n=1}^{\infty} \sqrt{1 - G(n)} < \infty$, it can be shown that a simple moment condition on the distribution $G$ implies this condition. Indeed, it is easily seen that $1 - G(n) \leq \frac{1}{(n+1)^{2 + \epsilon}} \sum_{j=1}^\infty g_j j^{2 + \epsilon}$ for all $n \in \bbn$ and $\epsilon > 0$, so that the following result immediately follows.

\begin{Lemma}\label{LemmaRegularVarying}
Let $X$ be a rv with distribution $G$. If there exists an $\epsilon > 0$ with $\bbe[ X^{2 + \epsilon}] < \infty$, then $\sum_{n=1}^{\infty} \sqrt{1 - G(n)} < \infty$.
\end{Lemma}

\subsection{Joint Asymptotic Normality of $\widehat{H}_n$ and $\hat{c}_n$}

\begin{Theorem}\label{TheoremAsymptoticVarianceHn}
Let $ x_1, \dots, x_l \in \bbn$, $ l \in \bbn$ and let the conditions of Theorem \ref{TheoremAsymptoticNormailityG_n} be satisfied. Then \begin{align*}
\sqrt{n} \begin{pmatrix}
\hat{c}_n - c \\ \widehat{H}_n(x_1) - H (x_1) \\ \vdots \\ \widehat{H}_n(x_l) - H (x_l) 
\end{pmatrix}  \stackrel{\mathcal{D}}{\rightarrow} \mathcal{N} (0 , \mathbf{T}'), 
\end{align*}
 
where the entries $\tau'_{i,j}$ of $\mathbf{T}'$ correspond to those given in Theorem \ref{TheoremAsymptoticNormailityG_n} in the following way: $\tau'_{1,1}  = c(1-c)$, $\tau'_{1,i+1} = \tau_{1,x_i}$ and $\tau'_{i+1,k+1} = \tau_{x_i, x_k}$. 
\end{Theorem}

\begin{proof}
First, let us define  $ \left(\widehat{H}_n(x_1) , \dots, \widehat{H}_n(x_l)  \right)^{T}  := \mathbf{H}_n$, $ \left(H (x_1) , \dots, H(x_l)  \right)^{T}  := \mathbf{H}$  and
\begin{align*}
\eta_i (k) :=   \frac{D(i)  \left(   \mathbf{1}_{\{ Z(i) \leq k \} } - H(k) \right)}{ \bbe[ D(0) ] },   
\end{align*}
notice that $\bbe[\eta_i (k)] = 0 $ for all $i \in \bbn_0$, $k \in \bbn$ by definition, see \eqref{DefinitionFunctionH}. We expand, for each $k \in \bbn$,
\begin{align*}
\widehat{H}_n(k) - H(k) = \frac{      \frac1n \sum_{i=1}^{n} D(i)  \mathbf{1}_{\{ Z(i) \leq k \} }  - H(k) \frac1n   \sum_{i=1}^{n} D(i)    } {\frac1n \sum_{i=1}^{n} D(i)},
\end{align*}
which leads to 

 \begin{align*}
&\sqrt{n} \left( \mathbf{H}_n - \mathbf{H} \right)  
= \sqrt{n} \frac{\bbe[ D(0) ]}{\frac1n \sum D(i)} \begin{pmatrix}
\frac1n \sum_{i=1}^{n} \eta_i (x_1) \\ \vdots \\ \frac1n \sum_{i=1}^{n} \eta_i (x_l) \end{pmatrix}.  \numberthis \label{eqJointNormalityExpansion}
 \end{align*}

Now, let $(t_0, t_1, \dots , t_l) \in \bbr^{l+1}$ and consider the sequence of rvs \[ 
\vartheta_i  := t_0 \left(  \mathbf{1}_{ \left\{ A(i) = 0 \right\}} - c \right) + \sum_{j=1}^{l} t_j \left( \eta_i (x_j)\right).
\]
It is obviously stationary and ergodic and has finite second moments by \eqref{LemmaWaldEquation}, we now show that this sequence satisfies condition \eqref{eqBillingsleyCondition}. Clearly, $ \sigma \left(\vartheta_i ; i \leq k \right) \subset \sigma \left( \xi(i) ; i \leq k \right) = \mathcal{F}_k$. By (19.25) in \cite{billingsley1999convergence} and the triangle inequality of the $L^{2}$-norm we calculate an upper bound for \eqref{eqBillingsleyCondition}: 
\begin{align*}
&\sum_{i =1}^{\infty} \left\| \bbe \left[ \vartheta_i | \mathcal{F}_0 \right]  \right\| \leq |t_0| \sum_{i =1}^{\infty} \left\| \bbe \left[  \mathbf{1}_{ \left\{ A(i) = 0 \right\}} - c  | \mathcal{F}_0 \right]  \right\|  + \sum_{i=1}^{\infty} \sum_{j=1}^{l} |t_j| \left\| \bbe \left[ \eta_i (x_j) | \mathcal{F}_0\right] \right\|.
\end{align*}
Absolute convergence of the latter series is ensured by the proof of Theorem \ref{NormalityD(i)1} and Corollary \ref{NormalityD(i)} and the triangle inequality in $L^{2}$. For the former series, notice that since $A(i)$ is independent of $\mathcal{F}_0$ for all $i >0$ by construction, $ \left\| \bbe \left[  \mathbf{1}_{ \left\{ A(i) = 0 \right\}} - c  \Big| \mathcal{F}_0 \right]  \right\| = 0$ for $i>0$. This implies finiteness of the entire expression, and in conclusion that $\frac{1}{\sqrt{n}} \sum_{i=1}^{n} \vartheta_i$ is asymptotically normal. 

For the calculation of the asymptotic variance we use the fact that $\bbe[ \eta_i (x_{j})] = 0$ for any $i\in \bbn_0$, $j \in \bbn$ as well as the independence of  $\mathbf{1}_{ \left\{ A(i) = 0 \right\}}$ and $\eta_0 (x_j)$ for all $x_j$ and $i > 0$. Straightforward algebra yields

\begin{align*}
&\sum_{j_1, j_2=1}^{k} t_{j_1} t_{j_2} \Bigg[ \bbe \left[ \eta_0 (x_{j_1})  \eta_0 (x_{j_2}) \right] +  \sum_{i=1}^{\infty} \left( \bbe \left[ \eta_0 (x_{j_1})  \eta_i (x_{j_2}) \right] + \bbe \left[ \eta_i (x_{j_1})  \eta_0 (x_{j_2}) \right] \right) \Bigg] \\
&+ \sum_{j_1 = 1}^{k} t_{j_1} t_0 \left[ \bbe \left[ \eta_0 (x_{j_1}) \mathbf{1}_{ \left\{ A(0) = 0 \right\}} \right] +  \sum_{i=1}^{\infty}  \bbe \left[ \eta_i (x_{j_1})  \mathbf{1}_{ \left\{ A(0) = 0 \right\}} \right]  \right] + t_0^{2} \Var \left(\mathbf{1}_{ \left\{ A(0) = 0 \right\}} \right), 
\end{align*} which is easily seen to be the variance of the random variable $(t_0, t_1, \dots , t_l ) \cdot \mathbf{X}$, where $\mathbf{X} \sim \mathcal{N} (0, \mathbf{T}')$. As we chose $(t_0, t_1, \dots, t_l) \in \bbr^{l}$ arbitrarily, we may apply the Cram\'er-Wold device to show that $\frac{1}{\sqrt{n}} \sum_{i=1}^{n} \vartheta_i$ converges weakly to a multivariate normal distribution with zero mean and covariance matrix given by $\mathbf{T}'$. By \eqref{eqJointNormalityExpansion} $\mathbf{H}_n$ is the product of a weakly convergent sequence and a sequence which converges almost surely to the constant $1$ by the ergodic theorem (cf. discussion following \eqref{PropRelationGH}). Application of Slutsky's Lemma concludes the proof.
\end{proof}

\begin{Remark}
The apparent differences in the expressions for the asymptotic covariances given in Theorem \ref{TheoremAsymptoticNormailityG_n} and the proof of Theorem \ref{TheoremAsymptoticVarianceHn} are only a matter of notation, as the stationarity of the sequences involved ensures that, e.g. 
\[ \bbe \left[ \eta_0 (k)  \eta_0 (m) \right] +  \sum_{i=1}^{\infty} \left( \bbe \left[ \eta_0 (k)  \eta_i (m) \right] + \bbe \left[ \eta_i (k)  \eta_0 (m) \right] \right) = \sum_{i= -\infty}^{\infty} \bbe \left[ \eta_0 (k)  \eta_i (m) \right].
\]
\end{Remark}

\section{Proof of the Functional Central Limit Theorem}
\label{SectionFLT}
In this section, we present the proof of Theorem \ref{TheoremAsymptoticNormailityG_n}. The previous section established the joint weak convergence of the finite-dimensional distributions. It thus remains to be seen that the sequence of estimators introduced in Section \ref{SectionIntroduction} is tight. The following theorem provides an important technical result for the proof of tightness. Note that we do not have to make any further assumptions about our model for the assertion to hold. Thus, our result remains applicable to a wide range of models.

\begin{Theorem}\label{TheoremSummabilityVarianceHn}
Let the conditions of Theorem \ref{TheoremAsymptoticNormailityG_n} be satisfied.  Then
$
\sum_{x=1}^{\infty} \tau^{2}_{x,x} 
$
converges absolutely. 
\end{Theorem}

\begin{proof}
Let $\eta_i(x)$ be defined as in the proof of Theorem \ref{TheoremAsymptoticVarianceHn} for $i \in \bbn_0$, $x \in \bbn$. Due to the stationarity of the $\eta_i(x)$'s, we have 
\begin{align*}
\tau^{2}_{x,x} = \sum_{i = - \infty}^{\infty} \bbe[ \eta_0 (x) \eta_i (x)] = \bbe[ \eta_0 (x)^{2} ] + 2 \sum_{i = 1}^{\infty} \bbe[ \eta_0 (x) \eta_i (x)],
\end{align*}
so that it suffices to show that $ \sum_{x=1}^{\infty} \sum_{i = 0}^{\infty} \bbe[ \eta_0 (x) \eta_i (x)]$ converges absolutely. Using Schwarz's inequality, we have
\begin{align}
| \bbe \left[ \eta_0 (x)  \eta_i (x)  \right]| \leq  \bbe \left[| \eta_0 (x)  | \cdot  | \bbe \left[ \eta_i (x) | \mathcal{F}_0 \right] | \right] \leq \| \eta_0 (x) \| \cdot \|  \bbe \left[ \eta_i (x) | \mathcal{F}_0 \right]\| \label{eqInequalityCauchySchwarzfromBillingsley}.
\end{align}
Using stationarity again, $\| \eta_0 (x) \| = \| \eta_1 (x) \| = \Var^{\frac12} ( \eta_1 (x))$, as $\bbe[\eta_1 (x)] = 0$. The variance corresponds exactly to the case $i=1$ and $y=x$  in Lemma \ref{LemmaInequalityVariancesTechnical}, so that
\begin{equation*}
\bbe[D(0)]^{2}  \bbe[ \eta_0 (x)^{2} ] = \bbe[D(0)]^{2}  \| \eta_0 (x) \|^{2} \leq  \frac{\bbe[A(0)^{2}]  K}{1-c} c^{x},
\end{equation*} 
as $ c^{2x}(1 - G(x))^{2} - 2 c^{2x}(1 - G(x)) + c^{x+i-1} = c^{2x}( G(x)^{2} -1) + c^{x} \leq c^{x} $ for each $i,x \in \bbn$. By the Weierstra\ss ~M-test, $\sum_{x=1}^{\infty}   \bbe[ \eta_0 (x)^{2} ]$ converges absolutely, and furthermore the term $\| \eta_0 (x) \| $ is uniformly bounded for each $x \in \bbn$. In order to complete the proof, we thus have to show absolute convergence of $
\sum_{x=1}^{\infty} \sum_{i=1}^{\infty} \| \bbe \left[ \eta_i (x) | \mathcal{F}_0 \right]\|$.

Just as in the proof of Theorem \ref{NormalityD(i)1}, we attempt to separate the probabilistic part of the conditional expectation from the deterministic part. Since we will be using the result of Lemma \ref{LemmaRecursionZi}, we use $\mathbf{1}_{\{ Z(i) \leq x \} } = 1 - \mathbf{1}_{\{ Z(i) > x \} }$ and calculate, for each $i,x \in \bbn$,
\begin{align*}
&\bbe \left[ D(i)   \left( H (x) - \mathbf{1}_{\{ Z(i) \leq x \} } \right)  | \mathcal{F}_0 \right] \\
&= \sum_{j=1}^{i-1} \bbe \left[ \left. \sum_{l=1}^{A(i-j)} \mathbf{1}_{\{ S_{i-j,l } = j \}}   \mathbf{1}_{\{ Z(i) > x \} }  \right| \mathcal{F}_0 \right]  -  (1 - H(x)) \sum_{j=1}^{i-1} \bbe \left[  \sum_{l=1}^{A(i-j)} \mathbf{1}_{\{ S_{i-j,l } = j \}}  \right] \\
&\qquad + \left( \bbe \left[   \mathbf{1}_{\{ Z(i) > x \} }  \big| \mathcal{F}_0 \right] -  (1 - H(x)) \right)  \sum_{j=i}^{\infty} \sum_{l=1}^{A(i-j)} \mathbf{1}_{\{ S_{i-j,l } = j \}}, \numberthis \label{eqAsyVarianceSecondStepA}
\end{align*}
where we used independence or measurability of the random variables $(A(t)_{t \in \bbz}$ and $(S_{j, \cdot})_{j \in \bbz}$ with respect to $\mathcal{F}_0$, as the case may be. Now, the second term in \eqref{eqAsyVarianceSecondStepA} is deterministic, and for the first term Lemma \ref{LemmaRecursionZi} implies that this conditional expectation is also deterministic. For the third term, we need to consider the cases $ i > x$ and $ i \leq x$ separately.

Let us first assume $ i > x$. Lemma \ref{LemmaRecursionZi} yields $\bbe \left[   \mathbf{1}_{\{ Z(i) > x \} }  \big| \mathcal{F}_0 \right] = c^{x}$, thus in this case \eqref{eqAsyVarianceSecondStepA} and the law of total expectation imply
\begin{align*}
& \bbe[D(0)] \cdot \| \bbe \left[ \eta_i (x) | \mathcal{F}_0 \right] \|   = c^{x} G(x)  \Var^{\frac{1}{2}} \left( \sum_{j=i}^{\infty} \sum_{l=1}^{A(i-j)} \mathbf{1}_{\{ S_{i-j,l } = j \}} \right),
\end{align*}
noticing that $1 - H(x) = c^{x}(1 - G(x))$ by \eqref{PropRelationGH}. Using the inequality \eqref{eqVarianceInequalityWald} for this variance, we find the upper bound $c^{x}G(x) \sqrt{\bbe[A(0)^{2}]} \sqrt{1 - G(i-1)}$ for each $i > x$. 

For the case $ i \leq x$, Lemma \ref{LemmaRecursionZi} implies that $\bbe \left[   \mathbf{1}_{\{ Z(i) > x \} }  \big| \mathcal{F}_0 \right] = c^{i-1} \mathbf{1}_{ \{  Z (1) > x - i +1 \} }$. Using  \eqref{eqAsyVarianceSecondStepA} as well as the law of total expectation again, we find
\begin{align*}
& \bbe[D(0)] \cdot \|\bbe \left[ \eta_i (x) | \mathcal{F}_0 \right]\|  \\
&= \Var^{\frac12} \left(\left( c^{i-1} \mathbf{1}_{ \{  Z(1) > x - i +1 \} } -  c^{x}(1 - G(x)) \right) \sum_{j=i}^{\infty} \sum_{l=1}^{A(i-j)} \mathbf{1}_{\{ S_{i-j,l } = j \}} \right).
\end{align*}
We find an upper bound for this variance by applying Lemma \ref{LemmaInequalityVariancesTechnical} with $y = x$, we further notice that $ c^{2x}(1 - G(x))^{2} - 2 c^{2x}(1 - G(x)) + c^{x+i-1} = c^{2x} G(x)^{2} + c^{x+i-1} - c^{2x} \leq ( c^x G(x) + c^{x/2})^2$. Combining the cases $i \leq x$ and $i > x$ and changing the summation index for convenience, we find
\begin{align*}
&\bbe[D(0)] \sum_{i=1}^{\infty}  \| \bbe \left[ \eta_i (x) | \mathcal{F}_0 \right] \| \leq \sqrt{ \frac{\bbe[A(0)^{2}]  K}{1-c}} \left( \sum_{i=0}^{\infty} \sqrt{1 - G(i) } \right) \left[ c^{x} G(x) + c^{\frac{x}{2}} \right],
\end{align*}
and as $c \in (0,1)$ by assumption, it follows that $\sum_{x=1}^{\infty} \sum_{i=1}^{\infty} \| \bbe \left[ \eta_i (x) | \mathcal{F}_0 \right]\|$ converges. Absolute convergence of the expression is clear from the positivity of the variance, and the proof is concluded by applying the Weierstra\ss ~M-test.
\end{proof}

\subsection{Functional Central Limit Theorems for $\mathcal{H}_n$}
\label{SectionFCLTandC0}

Recall the notations of  $\mathcal{H}_n, \mathcal{H}, \mathcal{G}_n$ and $\mathcal{G}$ as given in Section \ref{SectionModelandNotation}. In this section, we are concerned with asymptotic theorems in the separable Banach space $c_0$ of all sequences $x = (x_k)_{k \in \bbn}$ converging to zero, equipped with the norm $\| x \|_{c_0} = \sup_{k \in \bbn} |x_k|$. Following the construction in \cite{henze1996}, the Borel-$\sigma$-algebra of $c_0$ coincides with the smallest $\sigma$-algebra on $c_0$ such that the projections $x \mapsto \pi_k \circ x := x_k$, $k \in \bbn, x \in c_0$ are measurable. We denote this $\sigma$-algebra by $\mathcal{B}$. Now, for each $n \in \bbn$, $\sqrt{n} ( \mathcal{H}_n - \mathcal{H})$ is a mapping from $\{ \bbn^{\bbn} \cup \{0 \} \}^{\bbz}$ to $c_0$. Notice that the condition $\lim_{k \to \infty} \sqrt{n} ( \widehat{H}_n (k) - H  (k) ) = 0$ is fulfilled as both $\widehat{H}_n (\cdot) $ and $H (\cdot) $ are cdfs, thus both tend to 1. With straightforward topological argumentation, it can be shown that $\pi_k \circ ( \sqrt{n} ( \mathcal{H}_n - \mathcal{H})) = \sqrt{n} (\widehat{H}_n (k) - H(k))$ is a random variable and it follows that $\sqrt{n} ( \mathcal{H}_n - \mathcal{H})$ is $\mathcal{F}_\infty$-$\mathcal{B}$-measurable, implying that $\bbp \circ ( \sqrt{n} ( \mathcal{H}_n - \mathcal{H}))^{-1}$ is a Borel probability measure on $c_0$. In order to show convergence in distribution of a sequence to a random element in $c_0$, it is well known that we need to prove the weak convergence of finite-dimensional distributions and verify that the sequence is tight (see, e.g.,~\cite{billingsley1999convergence}).

In this section, expressions of the type $\mathcal{H} - 1$ should be interpreted as $ \left(  H (k)  - 1 \right)_{k \in \bbn}$ while scalar multiplication of the form $a \mathcal{H}$ denotes the sequence $ \left( a H (k)  \right)_{k \in \bbn}$.  With this notation, the process we are interested in at first is an element of $\bbr \times c_0$ given by 
$ \sqrt{n} [ ( \hat{c}_n, \ \mathcal{H} _n  -1   )  - (c,\ \mathcal{H} -1  ) ],
$ we choose this representation in order to make the application of the functional delta method at the end of this section more transparent. Analogous to the argumentation above, this process is a random element of $\bbr \times c_0$, and the following result shows that this process converges weakly to a limiting random element $\mathcal{W} \in \bbr \times c_0$.

\begin{Theorem}\label{AsymptoticNormalityH_n}
Let the conditions of Theorem \ref{TheoremAsymptoticNormailityG_n} be satisfied. Then there exists a Gaussian element $\mathcal{W}  = \left( w,\left( W_k \right) _{ k \in \bbn} \right)$ in $\bbr \times c_0$ such that $\bbe[w] = \bbe\left[ W_k \right] = 0$ and $\bbe\left[ W_k  W_m \right] = \tau^{2}_{k,m}$ , $\bbe\left[ w  W_m \right] = \tau^{2}_{1,m}$ as given in Theorem \ref{TheoremAsymptoticNormailityG_n} as well as $\bbe\left[ w^{2} \right] = c(1-c)$. 

Moreover,
$
 \sqrt{n} [ ( \hat{c}_n, \ \mathcal{H} _n  -1   )  - (c,\ \mathcal{H} -1  ) ] \stackrel{\mathcal{D}}{\rightarrow} \mathcal{W}.
$
\end{Theorem}
\begin{proof}
The convergence of the finite-dimensional distributions was shown in Theorem \ref{TheoremAsymptoticVarianceHn}. The tightness of the sequence remains to be established. Since marginal tightness implies joint tightness, it suffices to show that the sequences $\sqrt{n}  ( \hat{c}_n  - c )$ and $ \sqrt{n} ( \mathcal{H} _n  - \mathcal{H}  )$ are tight. For the former this assertion is obvious, as the random variables $\mathbf{1}_{\left\{ A(i) = 0 \right\} }$ are i.i.d. with bounded variance and the classical CLT yields tightness of $\sqrt{n}  \left( \hat{c}_n  - c \right)$.

For the latter sequence, we use Lemma 2.1. of \cite{henze1996} which provides two sufficient and necessary conditions for tightness in $c_0$. The first condition necessitates that for each positive $\delta$ and $l \in \bbn$ there is a finite constant $M$, such that 
\begin{equation}\label{ConditionTightness1}
\bbp \left( \left| \sqrt{n}\left(\widehat{H}_n(l) - H (l)   \right) \right| \leq M \right) \geq 1 - \delta, \quad n \geq 1.
\end{equation}
By Theorem \ref{TheoremAsymptoticVarianceHn} it follows that for each $l \in \bbn$, $\sqrt{n}(\widehat{H}_n(l) - H (l)   )$ converges weakly to a normal distribution with zero mean and bounded variance (cf. Theorem \ref{TheoremSummabilityVarianceHn}), thus condition \eqref{ConditionTightness1} is established.

The second condition is satisfied if for each positive numbers $\delta, \epsilon$ there exist integers $n_0$ and $l_0$ such that 
\begin{equation}\label{ConditionTightness2}
\bbp \left( \sup_{k \geq l_0} \sqrt{n} \left| \widehat{H}_n (k) - H(k) \right| > \epsilon \right) \leq \delta \text{ for all } n \geq n_0.
\end{equation}
The proof of the second condition is provided below, it is divided into three steps to make it more comprehensible. \medskip

\textbf{Step 1:} We show that, with the notation of the proof of Theorem \ref{TheoremAsymptoticVarianceHn} and $S_n(k) :=  \sum_{i=1}^{n}  \eta_i(k) $, it holds that $\sum_{k=1}^{\infty} \frac1n \bbe \left[ S_n^{2} (k) \right] $ is finite. For this, let us recall that in the mentioned proof we showed that the stationary and ergodic sequence  $\left( \eta_i(k) \right)_{i \in \bbz}$ satisfies the conditions of Theorem 19.1.~in \cite{billingsley1999convergence}. There, it is shown that for each $n$ (cf. \cite{billingsley1999convergence}, (19.11)):
\begin{align*}
 \left|  \tau_{k,k} - \frac1n \bbe \left[ S_n^{2} (k) \right] \right| \leq 2 \sum_{l=n}^{\infty} | \bbe \left[ \eta_0(k) \eta_l(k) \right] | + \frac2n \sum_{i=1}^{n-1} \sum_{l=i}^{\infty} |  \bbe \left[ \eta_0(k) \eta_l(k) \right] |. 
\end{align*} Further, 
$
|  \bbe \left[ \eta_0(k) \eta_l(k) \right] | \leq \left\| \eta_0 (k) \right\| \left\| \bbe \left[ \eta_l(k) | \mathcal{F}_0 \right] \right\|$, cf. \eqref{eqInequalityCauchySchwarzfromBillingsley} 
 and, by the discussion following that equation, $ \bbe[D(0)] \left\| \eta_0 (k) \right\| \leq \sqrt{(\bbe[A(0)^{2}] K)/(1-c)}$ for all $k \in \bbn$. Both inequalities together with Theorem \ref{TheoremSummabilityVarianceHn} yield
\begin{align*}
&\left|\sum_{k=1}^{\infty} \tau_{k,k} - \sum_{k=1}^{\infty} \frac1n \bbe \left[ S_n^{2} (k) \right] \right| \\
&\leq   \frac{2}{\bbe[D(0)]} \sqrt{\frac{\bbe[A(0)^{2}] K}{1-c}} \left( \sum_{k=1}^{\infty} \sum_{l=n}^{\infty}  \left\| \bbe \left[ \eta_l(k) | \mathcal{F}_0 \right] \right\| + \frac1n \sum_{k=1}^{\infty}\sum_{i=1}^{n-1} \sum_{l=i}^{\infty}  \left\| \bbe \left[ \eta_l(k) | \mathcal{F}_0 \right] \right\|  \right) .
\end{align*}
Since $\sum_{k=1}^{\infty} \sum_{l=1}^{\infty}  \left\| \bbe \left[ \eta_l(k) | \mathcal{F}_0 \right] \right\|$ converges absolutely as shown in the proof of Theorem \ref{TheoremSummabilityVarianceHn}, this expression is finite. As $\sum_{k=1}^{\infty}  \tau^{2}_{k,k}$ is finite by the same result, this implies finiteness of $\sum_{k=1}^{\infty} \frac1n \bbe \left[ S_n^{2} (k) \right] $ and the convergence of the latter to the former, since the entire expression tends to $0$ for $n \to \infty$. For the former series this is obvious, for the latter this follows from the dominated convergence theorem. \medskip

\textbf{Step 2:}
For any $\xi > 0$ we denote the event $\left\{  \left| \frac1n \sum_{i=1}^{n} D(i) - \bbe[D(0)] \right| \leq \xi \right\} $ by $A(n,\xi)$ and the complementary event by $A^{c}(n,\xi)$. We observe that
\begin{align*}
&\bbe \left[\left(\widehat{H}_n(k) - H (k)   \right)^{2}  \Big|  A^{c}(n,\xi) \right] \leq \bbe \left[ \widehat{H}_n(k) - H (k)     \Big|  A^{c}(n,\xi) \right] \leq 1 - H(k), 
\end{align*}
we used that $ \bbe[\widehat{H}_n(k)  |  A^{c}(n,\xi) ] \leq 1$, as we have $\widehat{H}_n(k) \leq 1$ for each $k,n \in \bbn$ by definition (cf. \eqref{DefinitionFunctionH}). Further, we used that $ | \widehat{H}_n(k) - H(k) | \leq 1$ as both $0 \leq \widehat{H}_n(k) \leq 1$ and $0 \leq H(k) \leq 1$ hold for all $k,n \in \bbn$. 

Now let us consider the behavior of $\bbp \left( A^{c}(n,\xi) \right)$ as $n$ increases, we follow the idea of \cite[Theorem 11]{kachurovskii1996}. Using Chebyshev's inequality, 
$n  \bbp \left( A^{c} (n, \xi) \right)  \leq  \Var \left(   \sum_{i=1}^{n} D(i) \right)/(\xi^2 n )$. Recalling the results of Corollary \ref{NormalityD(i)} and the definition of $\sigma^2$ given therein, the ergodic and stationary sequence $(D(t))_{t \in \bbz}$ satisfies the conditions of Theorem 19.1.~and hence necessarily the inequality (19.11) in \cite{billingsley1999convergence} with
\begin{align*}
\Var \left( \frac{1}{\sqrt{n}}  \sum_{i=1}^{n} D(i) \right) \leq \sigma^{2} + \underbrace{2 \sum_{l=n}^{\infty} \Cov \left( D(0), D(l) \right) + \frac{2}{n} \sum_{i=1}^{n-1} \sum_{l=i}^{\infty} \Cov \left( D(0), D(l) \right)}_{=: C(n)},  
\end{align*}
notice that $\lim_{n \to \infty} C(n) = 0$ using analogous argumentation as in Step 1. Combining these findings, we have
\begin{align}\label{eqVarianceHnInequalitySecondExpansion}
n \bbp \left( A^{c}(n,\xi) \right)  \bbe \left[\left(\widehat{H}_n(k) - H (k)   \right)^{2}  \Big|  A^{c}(n,\xi) \right] \leq \frac{\sigma^{2} + C(n)}{\xi^{2}} \left( 1 - H(k) \right).  
\end{align}

Now, for each $k \in \bbn$, $\sqrt{n}(\widehat{H}_n(k) - H (k)  )$ is equal to (cf. \eqref{eqJointNormalityExpansion}),
\[ \frac{1}{ \sqrt{n}} S_n (k) \left( 1 + \bbe[D(0)]\left( \frac{1}{\frac1n \sum_{i=1}^{n} D(i)} - \frac{1}{\bbe[ D(0) ]} \right) \right) =  \frac{1}{ \sqrt{n}} S_n (k) \left( \frac{\bbe[D(0)]}{\frac1n \sum_{i=1}^{n} D(i)} \right).
\]

 With this, for each $k \in \bbn$,
\begin{align*}
&  \bbe \left[ \frac{1}{ n} S^{2}_n (k) \left( \frac{\bbe[D(0)]}{\frac1n \sum_{i=1}^{n} D(i)} \right)^{2}  \right] = \bbe \left[ \frac{1}{ n} S^{2}_n (k) \left( \frac{\bbe[D(0)]}{\frac1n \sum_{i=1}^{n} D(i)} \right)^{2}   ( \mathbf{1}_{\{ A(n, \xi) \} } + \mathbf{1}_{\{ A^{c}(n, \xi) \} } ) \right] \nonumber  \\
 &\leq \left( \frac{\bbe[D(0)]}{\bbe[D(0)] - \xi} \right)^{2} \bbe \left[ \frac{1}{ n} S^{2}_n (k) \right] + \bbp \left( A^{c}(n,\xi) \right) n \  \bbe \left[\left(\widehat{H}_n(k) - H (k)   \right)^{2}  \Big|  A^{c}(n,\xi) \right] \nonumber\\
 &\stackrel{\eqref{eqVarianceHnInequalitySecondExpansion}}{=} \left( \frac{\bbe[D(0)]}{\bbe[D(0)] - \xi} \right)^{2} \bbe \left[ \frac{1}{ n} S^{2}_n (k) \right] + \frac{\sigma^{2} + C(n)}{\xi^{2}} \left( 1 - H(k) \right) . 
\end{align*} \medskip

\textbf{Step 3:}
For the final step, let $\epsilon, \delta > 0$. Then, with Markov's inequality,

\begin{align*}
&\bbp \left( \sup_{k \geq l} \sqrt{n} \left|  \widehat{H}_n (k) - H(k) \right| > \epsilon \right) \leq \frac{1}{\epsilon^{2}} \sum_{k \geq l}  \bbe \left[ \frac{1}{ n} S^{2}_n (k) \left( \frac{\bbe[D(0)]}{\frac1n \sum_{i=1}^{n} D(i)} \right)^{2}  \right] \\
&\leq \frac{1}{\epsilon^2} \sum_{k \geq l} \left[ \left( \frac{\bbe[D(0)]}{\bbe[D(0)] - \xi} \right)^{2} \bbe \left[ \frac{1}{ n} S^{2}_n (k) \right] + \frac{\sigma^{2} + C(n)}{\xi^{2}} \left( 1 - H(k) \right) \right],
\end{align*} 
we used Step 2 for the final inequality. By Step 1, the series $\sum_{k=1}^{\infty} \frac1n \bbe \left[ S_n^{2} (k) \right] $ converges, for the series $\sum_{k=1}^{\infty} (1 - H(k)) $ convergence follows from \eqref{PropRelationGH}. Let us set $n_0 := \sup_{k \in \bbn} C(k)$, which is possible as $\lim_{n\to \infty} C(n) = 0$. As $\xi$ is arbitrary but fixed, it is clear that we are able to find $l_0 \in \bbn$ such that the expression above is less than $\delta$  for  $l\geq l_0$ and $n \geq n_0$. Thus, \eqref{ConditionTightness2} is satisfied as $\epsilon, \delta >0$ were chosen arbitrarily. Since both conditions \eqref{ConditionTightness1} and \eqref{ConditionTightness2} combined yield tightness of the sequence $ \sqrt{n} ( \mathcal{H} _n  - \mathcal{H}  )$, this concludes the proof. 
\end{proof}

\subsection{The Functional Delta Method}

In the preceding sections, we established the asymptotic behavior of the estimator $\mathcal{H}_n$. However, our original goal was the estimation of the service time distribution $G$, by \eqref{PropRelationGH} this is done via $\widehat{G}_n (x) := 1 - \hat{c}_n^{-x} (1 - \widehat{H_n} (x))$. Let us define the mapping 
\begin{align*}
\phi : \bbr \times c_0 &\rightarrow \bbr^{\bbn} \\
 \left( a , (x_k)_{ k \in \bbn}   \right)  & \mapsto  \left( x_k a^{-k} \right) _{ k \in \bbn.}
\end{align*}

Notice that $\phi$  maps to $\bbr^{\bbn}$ and not to $c_0$, as for any sequence $  \left( x_i \right)_{i \in \bbn} \in c_0$ it does not hold in general that $\phi(a , (x_k)_{ k \in \bbn} )  \in c_0 $ for some $a \in \bbr$. However, it is clear that the set $\mathbb{D}_\phi  := \left\{ (a, (x_k)_{ k \in \bbn}) \in \bbr \times c_0 \, | \,  \phi(a, (x_k)_{ k \in \bbn}) \in c_0 \right\}$ is not an empty set. It follows easily from \eqref{PropRelationGH} that
$
\sqrt{n}  \phi  \left( \hat{c}_n, \left( \mathcal{H} _n -1  \right) \right) =  \sqrt{n} \left(  \mathcal{G}_n - 1 \right), 
$
and thus that 
\begin{align*}
\sqrt{n} \ \left[ \phi  \left( \hat{c}_n, \left( \mathcal{H} _n -1  \right) \right) - \phi  \left( c, \left( \mathcal{H} -1  \right) \right) \right]  =  \sqrt{n} \left(  \mathcal{G}_n - \mathcal{G} \right), 
\end{align*} 
our process of interest. 

Now, $  \sqrt{n} \left(  \mathcal{G}_n - \mathcal{G} \right) \in c_0$ (since both $\widehat{G}_n (\cdot)$ and $G (\cdot)$ are cdfs, thus both tend to 1). The remaining step lies in proving Hadamard differentiability of the mapping $\phi$. Denote $\theta = (c, H(1) -1, H(2) -1, \dots ) \in \bbr \times c_0$, it suffices to show (cf. \cite{vandervaart}, Theorem 20.8.) that the Hadamard derivative at $\theta$, denoted by $ \phi_{\theta}'$ exists on the subset $\mathbb{D}_\phi  \subset \bbr \times c_0$, notice that $\theta \in \mathbb{D}_\phi $. This is shown in the following Lemma.

\begin{Lemma}\label{LemmaHadamardDifferentiability}
Let the conditions of Theorem \ref{TheoremAsymptoticNormailityG_n} be satisfied. Then  $ \phi: \mathbb{D}_\phi  \mapsto c_0$ is Hadamard differentiable at $\theta$ tangentially to $\mathbb{D}_\phi$ and 
\begin{align*}
\phi_{\theta}' \left( w, \left( x_k \right) _{k \in \bbn}\right) = \left( \frac{x_k}{c^{k}} - k \frac{1 - H(k)}{c^{k+1}} w \right)_{k \in \bbn}
\end{align*}
\end{Lemma}

\begin{proof}
For any element $ (w, (x_k)_{ k \in \bbn}) \in \mathbb{D}_\phi $ we have that $\phi'_\theta (w, (x_k)_{ k \in \bbn}) \in c_0$. This follows from  $\lim_{ k \to \infty} c^{-k} x_k = 0$  and $(k (1-H(k)) w) / c^{k+1} = (k (1 - G(k) w)/c$, where due to $k (1 - G(k)) \leq \bbe[G] - \sum_{i=1}^{k} i g_i$ the latter expression tends to $0$. $\phi_{\theta}'$ as given is linear on $\mathbb{D}_\phi$, and as the projections $\pi_k$ of $\phi$ are continuous it follows that $\phi_{\theta}'$ is continuous on $\mathbb{D}_\phi$. Following \cite{vandervaart}, it remains to be seen that
$
\|[\phi( \theta +  t h_t) - \phi(\theta)] / t - \phi_{\theta}'(h)  \|_{c_0} \rightarrow 0$ as $ t \rightarrow 0$,
for every $h_t \rightarrow h$ such that $t h_t$ is contained in the domain of $\phi$ for all small $t > 0$ and such that $ h \in \mathbb{D}_\phi$. For convenience, we denote the first component of the vectors $h \in \bbr \times c_\phi$ by $h (0)$, the second by $h (1)$ and so on, similarly for $h_t$. We calculate 
\begin{align*}
&\left\| \frac{\phi( \theta +  t h_t) - \phi(\theta)  }{t} - \phi_{\theta}'(h) \right\|_{c_0} \\
&=  \sup_{ k\in \bbn}  \left|  \frac{ h_t (k)}{(c + t h_t (0) )^{k}}  + (1 - H(k)) \frac{(c + t h_t (0) )^{-k}- c^{-k}  }{t} - \frac{h(k)}{c^{k}} - k \frac{1 - H(k)}{c^{k+1}} h(0) \right|.
\end{align*}
Now, for each $k$ it quite obviously holds that 
\begin{align*}
\lim_{t \to 0} \frac{(c + t h_t (0) )^{-k}- c^{-k}  }{t} = \lim_{t \to 0} \frac{ c^{k}  - (c + t h_t (0) )^{k} }{t c^{k}(c + t h_t (0) )^{k} } = \lim_{ t \to 0} \frac{-k h_t (0)}{c^{k+1}}.
\end{align*}
With the assumptions made for $h,h_t$, the entire expression above tends to $0$ for $t \to 0, h_t \to h$, concluding the proof. 
\end{proof}

Thus, using Lemma \ref{LemmaHadamardDifferentiability} together with Theorem \ref{AsymptoticNormalityH_n} and the Delta Method for Hadamard-differentiable functions (cf. \cite[Theorem 20.8]{vandervaart}) yields that 
\[ \sqrt{n} \left[ \phi( \hat{c}_n, \mathcal{H}_n - 1 ) - \phi(c, \mathcal{H} - 1) \right] = \sqrt{n} \left( \mathcal{G}_n - \mathcal{G} \right) \stackrel{\mathcal{D}}{\rightarrow} \phi'_\theta (\mathcal{W}) = \mathcal{V},
\]
whereby we have concluded the proof of Theorem \ref{TheoremAsymptoticNormailityG_n}.

\section{Bootstrap}
\label{Bootstrap}

The results of the previous sections are difficult to apply in practice as the covariance kernel established in Theorem \ref{TheoremAsymptoticNormailityG_n} is very involved and depends on the a priori unknown distributions $H$ or $G$ in a complicated manner. This problem was already pointed out in the original article introducing this estimation technique, where it is stated that (cf. \cite[p. 653]{brown}) "[w]ere one to verify the mixing conditions there would still remain the difficulty of computing the covariance kernel of the limiting process". We address this problem by showing that bootstrapping techniques are available for the estimation of the kernel of the limiting distribution and lead to correct results.

In our situation, we assume to be given data sets of  the form $\left( A(t) \right)_{t \in \{1, \dots, n \}}$ and $\left( D(t) \right)_{t \in \{1, \dots, n \}}$. The covariance kernel established in Theorem \ref{TheoremAsymptoticNormailityG_n} involves the parameter $c$. The estimator $\hat{c}_n$ is based on i.i.d.~observations of the arrival process and can be established independently from the estimators $\widehat{H}_n (\cdot)$. We suggest using $\hat{c}_n$ as plug-in estimate in the covariance kernel, we will not focus on this here.

Let us consider a bootstrap technique for the estimators $\widehat{H}_n (\cdot)$. Since we are not dealing with i.i.d.~random variables but rather with a stationary and ergodic sequence, we will apply the moving blocks bootstrap resampling procedure as first introduced by \cite{kunsch1989}. In order to do so, we fix a block size $b$ for a set of data of size $n$ and denote $k = \frac{n}{b}$. We introduce the random variables $\left\{ I_j \right\}_{j \in \{1, \dots, k\}}$ which are i.i.d.~uniformly distributed on $\{1, \dots, n-b+1\}$. The bootstrap sample is then given as $\{ (  X_{I_d +1}, X_{I_d +2}, \dots , X_{I_d +b} )\}_{d \in \{1, \dots, k\}}$. In order to keep the notation concise, we write $\bbp^{*}$, $\bbe^{*}$ and $\Var^{*}$ for the conditional probability, expectation and variance with respect to $\{ I_j \}_{j \in \{1, \dots, k\}}$, respectively, and we abbreviate the bootstrap sample by $\{ X^{*}_i \}_{i \in \{1, \dots, n\}}$. Expressions like $ \frac1n \sum_{i=1}^{n} X^{*}_i$ are thus shorthand for $\frac{1}{k} \sum_{d=1}^{k} \frac{1}{b} \sum_{i=1}^{b} X_{I_d + i}$ and so forth. First, we prove a more general result than needed for our specific purposes, as this general result is of interest on its own.

\begin{Theorem}\label{TheoremBootstrap}
Let $ \left( X_i \right)_{i \in \bbz}$ be a stationary and ergodic sequence of random variables satisfying \eqref{eqBillingsleyCondition}, $\bbe[X_0] = 0$ and $\bbe[X_0^{4}] < \infty$. Let $b= b(n)$ be a sequence of real numbers such that $b = o(n^{\alpha})$ with $\alpha \in (0, 2/5)$.

Then 
\begin{align*}
\sup_{ x \in \bbr} \left| \bbp \left[ \sqrt{n} \left( \frac1n \sum_{i=1}^{n} X_i \right) \leq x \right] - \bbp^{*} \left[ \sqrt{n} \left( \frac1n \sum_{i=1}^{n} X^{*}_i - \bbe^{*}[X^{*}_i] \right) \leq x \right] \right| \to 0 
\end{align*} in probability. 
\end{Theorem}

\begin{proof}
First note that under the assumptions on the sequence $b$ there exists a sequence $M$ such that $M \to \infty$, $b^{2}M^{2}/n \to 0$ and $b/M^{4} \to 0$ as $n\to \infty$. Our goal is the application of Theorem 3 in \cite{radulovic2012}. The first condition is satisfied, as the assumptions of Theorem 19.1.~in \cite{billingsley1999convergence} are satisfied, thus $ \Var \left( 1 / \sqrt{n} \sum_{i=1}^n X_i \right) \rightarrow \rho^{2} > 0$. Markov's inequality implies  $b \bbe\left[ X_1^{2} \mathbf{1}_{ \left\{  |X_1 | > M \right\}}  \right] \leq  b \bbe\left[ X_1^{4} \mathbf{1}_{ \left\{  |X_1 | > M \right\}}  \right] \leq \frac{b}{M^{4}} \bbe\left[ X_1^{4} \right]$ which converges to $0$ by assumption, proving the second condition.

For the third condition we show that with the notation  $ Y_{i,b} :=  (1 / \sqrt{b}) \sum_{j=i}^{b + i -1} X_j$, the expression  $(1/n-b+1) \sum_{i=1}^{n-b+1}  Y_{i,b}^{2}$ converges in probability. Since the $X_i$ (and thus the $Y_{i,b}$) are stationary, $
\bbe [ 1/(n-b+1)  \sum_{i=1}^{n-b+1} Y_{i,b}^{2} ] = 1/b \ \bbe [ ( \sum_{j=1}^{b} X_j )^{2} ]  \rightarrow \rho^{2},
$
as $b \rightarrow \infty$ by  \cite[(19.11)]{billingsley1999convergence}. Now, since $\bbe[X_i^{4}] < \infty$ it follows that $\bbe[Y_{i,b}^{4}] < \infty$ and due to the stationarity of the $Y_{i,b}$ we have 
\begin{align*}
 \Var \left[ \sum_{i=1}^{n-b+1} Y_{i,b}^{2} \right] &\leq (2b + 1) (n-b+1) \Var \left[ Y_{1,b}^{2} \right] + (n-b+1)  \sum_{i=b}^{n-b+1} \Cov \left(Y_{1,b}^{2}, Y_{i,b}^{2} \right),
\end{align*}
where we used the Cauchy-Schwarz inequality. Now, for any $i \in \bbn$ (including the case $i=1$), 
the multi-linearity of the covariance, the application of equation (13) in \cite{bohrnstedt1969} (notice that $\bbe[X_i] = 0$ for all $i$) and the stationarity of the $X_i$'s yield
\begin{align}
 &\Cov (Y_{1,b}^2, Y_{i,b}^2)  \nonumber \\
 &= 2 \left( \Cov \left( X_1, X_i \right) + \sum_{j=1}^{b-1} \frac{b - j}{b} \Cov \left( X_1, X_{i-j} \right)  +\sum_{j=1}^{b-1} \frac{b - j}{b} \Cov \left( X_1, X_{i+j} \right) \right)^2. \label{eqCovariancesBootstrap}
\end{align}

Now, since $\bbe[X_i] = 0$, $\Cov(X_1, X_i) = \bbe[X_1 X_i] \leq \| X_1 \| \cdot \left\| \bbe \left[ X_i | \mathcal{F}_1 \right] \right\|$ (cf. \eqref{eqInequalityCauchySchwarzfromBillingsley}). Thus, we first have 
\begin{align*}
\Var \left[ Y_{1,b}^{2} \right]  \leq 2 \| X_1 \|^{2} \left(  \| X_1 \|  +  2 \sum_{j=1}^{b-1} \left\| \bbe \left[ X_j | \mathcal{F}_1 \right] \right\|  \right)^{2},
\end{align*}
as $b \to \infty$ this expression is bounded by the assumption of this theorem. For the covariances we use the same argumentation, and discuss each summand in \eqref{eqCovariancesBootstrap} separately. Let $i = b + a$ for some $a \in \{0,1, \dots, n- 2b +1 \}$, then clearly $  \left\|\bbe \left[ X_{b + a} | \mathcal{F}_1 \right] \right\|$ tends to $0$. Further, by dominated convergence, 
\begin{align*}
\lim_{b \to \infty}  \sum_{j=1}^{b-1} \frac{b-j}{b} \left\| \bbe \left[ X_{b + a-j} | \mathcal{F}_1 \right] \right\| = \lim_{b \to \infty}  \sum_{j=a+1}^{b+a-1} \frac{j-a}{b} \left\| \bbe \left[ X_{j} | \mathcal{F}_1 \right] \right\| = 0,
\end{align*} for each value of $a$, as the series $\sum_{j=1}^{\infty} \left\| \bbe \left[ X_{j} | \mathcal{F}_1 \right] \right\|$ converges. Analogous argumentation applies to the third summand in \eqref{eqCovariancesBootstrap}. Combining all of these results we find that
\begin{align*}
\Var \left[ \frac{Y_{1,b}^{2}}{n-b+1}  \right] \leq \frac{2b+1}{n-b+1} \Var \left[ Y_{1,b}^{2} \right] + \frac{n-2b+1}{n-b+1} \sup_{i \in \{b, \dots, n-b+1\} } \Cov \left(Y_{1,b}^{2}, Y_{i,b}^{2} \right).
\end{align*} As $n \to \infty$ (and thus $b \to \infty$), the latter expression tends to $0$. For the former expression, notice that $b^{2} / n \to 0$, thus $(2b +1)(n-b+1) \to 0$. In conclusion, we showed that 
$ \bbe [ (1/n-b+1) \sum_{i=1}^{n-b+1} Y_{i,b}^{2} ] \to \rho^{2}$ and $\Var [ (1/n-b+1) Y_{1,b}^{2} ] \to 0$, implying convergence in probability. Thus, as argued in the proof of Theorem 2 in \cite{radulovic2012}, it follows that $ \sqrt{n} \left( \frac1n \sum_{i=1}^{n} X^{*}_i - \bbe^{*}[X^{*}_i] \right)$ converges to an asymptotically normal distribution $\mathcal{N} (0, \rho^{2})$ in probability. The assertion follows due to the continuity of the normal distribution. 
\end{proof}

\begin{Remark}
A general problem for the application of moving block bootstrap results is that there is no canonical choice for the block length $b$. The theoretical optimal block length in a different albeit similar situation to the one discussed here was calculated to the order $o(n^{1/3})$, we refer to \cite[Corollary 7.1.]{lahiri2003resampling}. The conditions of Theorem \ref{TheoremBootstrap} allow for such a choice of $b$ and we suggest to use this block length in practical applications. However, we point out to the authors' best knowledge there exists no result on the theoretical optimal block length for the specific situation discussed in this paper.
\end{Remark}

Returning to the problem at hand, we first construct the bootstrap samples of the processes $\{ D(t)  \mathbf{1}_{\{ Z (t)   \leq x \} }\}_{t \in \bbz}$ and $\{  D(t) \}_{t \in \bbz}$ using the scheme described above, we denote these samples by $\{ ( D(t)  \mathbf{1}_{\{ Z (t)   \leq x \} } )^{*}\}_{t \in \{1, \dots, n\}}$ and $\{ ( D(t) )^{*}\}_{t \in \{1, \dots, n\}}$. We now construct the bootstrap estimator as 

\begin{align*}
 \widehat{H}^{*}_n  (x) := \frac{  \sum_{i=1}^{n} \left( D(i)  \mathbf{1}_{\{ Z (i)   \leq x \} } \right)^{*} }{ \sum_{i=1}^{n} \left( D(i) \right)^{*}}   \text{  and  } H^{*} (x) := \frac{\bbe^{*} \left[ \left( D(0) \mathbf{1}_{\{ Z (0)  \leq x \} } \right)^{*} \right]}{\bbe^{*}[ \left(D(0) \right)^{*} ]}. 
\end{align*}
We further denote $(\widehat{H}^{*}_n(x_1) , \dots, \widehat{H}^{*}_n(x_l)  )^{T}  := \mathbf{H}^{*}_n$ and $ (H^{*} (x_1) , \dots, H^{*}(x_l) )^{T}  := \mathbf{H}^{*}$ in complete analogy with the notation $\mathbf{H}_n$ and $\mathbf{H}$ in the proof of Theorem \ref{TheoremAsymptoticVarianceHn}.

\begin{Corollary}
Let $ x_1, \dots, x_l \in \bbn$, $ l \in \bbn$. Let $b = b(n),M = M(n)$ be sequences of real numbers such that $b,M \rightarrow \infty$, $\frac{b^{2} M^{2}}{n} \rightarrow 0$ and $\frac{b}{M^{4}} \rightarrow 0$ as $n \to \infty$. Let the conditions of Theorem \ref{TheoremAsymptoticNormailityG_n} be satisfied and let $\bbe \left[ A(0)^{4} \right] < \infty$.

Then 
\begin{align*}
\sup_{ x \in \bbr^{l}} \left| \bbp \left( \sqrt{n} \left( \mathbf{H}_n - \mathbf{H} \right) \leq x \right) - \bbp^{*} \left( \sqrt{n} \left(\mathbf{H}^{*}_n - \mathbf{H}^{*}  \right) \leq x \right) \right| \to 0 
\end{align*} in probability.
\end{Corollary}

\begin{proof}
Let $(t_1, \dots , t_l) \in \bbr^{l}$. In the proof of Theorem \ref{TheoremAsymptoticVarianceHn} we showed that the stationary and ergodic sequence  $(X_i)_{ i \in \bbz}$ with $X_i :=  \sum_{j=1}^{l} t_j (\eta_i (x_j))  $ satisfies the assumptions of Theorem 19.1.~in \cite{billingsley1999convergence}. All other conditions of Theorem \ref{TheoremBootstrap} are satisfied, it remains to be seen that $\bbe[X_1^{4}] < \infty$. This, however is easily obtained from the finiteness of the fourth moment of $A(0)$. Application of the Cram\'er-Wold device as in the proof of Theorem \ref{TheoremAsymptoticVarianceHn} concludes the proof. 
\end{proof}

\section{Acknowledgements}

The authors would like to thank Prof.~Mark Podolskij and Prof.~Anne Leucht for helpful discussion during the early stages of this article. Furthermore, the authors would like to thank the Associate Editor and a referee for very useful comments and suggestions. In particular, one of the Associate Editor's remarks led to an improved statement of the condition in Lemma \ref{LemmaRegularVarying}.

The research of Sebastian Schweer and Cornelia Wichelhaus has been supported by the \emph{Deutsche Forschungsgemeinschaft} (German Research Foundation) within the programme "Statistical Modeling of Complex Systems and Processes - Advanced Nonparametric Approaches", grant GRK 1953.

\bibliographystyle{plain}



\end{document}